\documentclass[a4paper,12pt,reqno]{amsart}

\usepackage{amsmath, amssymb, amsthm, fullpage}
\usepackage{bbm}
\usepackage{enumerate, enumitem}
\usepackage{hyperref}
\usepackage{verbatim}
\usepackage{color}

\newtheorem{thm}{Theorem}[section]
\newtheorem{lemma}[thm]{Lemma}

\newtheorem{cor}[thm]{Corollary}
\newtheorem{prop}[thm]{Proposition}
\newtheorem{conj}[thm]{Conjecture}
\newtheorem{obs}[thm]{Observation}

\theoremstyle{definition}
\newtheorem{defn}[thm]{Definition}

\newcommand{\scolon}{:}

\newcommand{\cA}{\mathcal{A}}

\newcommand{\FF}{\mathcal{F}}
\newcommand{\GG}{\mathcal{G}}
\newcommand{\HH}{\mathcal{H}}

\newcommand{\UU}{\mathcal{U}}
\newcommand{\VV}{\mathcal{V}}

\newcommand{\GGn}{\GG_*}

\newcommand{\STOP}{\texttt{STOP}}

\newcommand{\Delf}[4]{\Delta_{(\ell_0{#1}, \ell_1{#2})}^{(i_0{#3},i_1{#4})}}
\newcommand{\Delp}{\Delta_{(\ell_0',\ell_1')}^{(i_0',i_1')}}
\newcommand{\Delnp}{\Delta_{(\ell_0,\ell_1)}^{(i_0',i_1')}}
\newcommand{\Del}{{\Delf{}{}{}{}}}
\newcommand{\MDel}{M^{(i_0,i_1)}_{(\ell_0,\ell_1)}}
\newcommand{\MDelp}{M^{(i_0',i_1')}_{(\ell_0,\ell_1)}}

\title{On the number of sets with a given doubling constant}
\author{Marcelo Campos
        }
\address{IMPA, Estrada Dona Castorina 110, Jardim Bot\^anico, Rio de Janeiro, RJ, Brasil}
\email{marcelo.campos@impa.br}
\thanks{Research partially supported by CNPq.}
\begin{document}
\begin{abstract}
We study the number of $s$-element subsets $J$ of a given abelian group $G$, such that $|J+J|\leq K|J|$. Proving a conjecture of Alon, Balogh, Morris and Samotij, and improving a result of Green and Morris, who proved the conjecture for $K$ fixed, we provide an upper bound on the number of such sets which is tight up to a factor of $2^{o(s)}$, when $G=\mathbb{Z}$ and  $K=o(s/(\log n)^3)$. We also provide a generalization of this result to arbitrary abelian groups which is tight up to a factor of $2^{o(s)}$ in many cases. The main tool used in the proof is the asymmetric container lemma, introduced recently by Morris, Samotij and Saxton.
\end{abstract}
\maketitle
\section{Introduction}
In additive combinatorics one of the main objectives of the field is, given an abelian group $G$ and a finite subset $A\subset G$, to understand the relation between the sumset $A+A$ and $A$.
In this direction, a fundamental result of Freiman \cite{Freiman59} says that for $G=\mathbb{Z}$, if $|A+A|\leq K|A|$ (we say that $A$ has \textit{doubling constant} $K$), then there is a generalized arithmetic progression $P$ such that $A\subset P$, the dimension of $P$ is at most $f(K)$, and $|P|\leq f(K)|A|$ for some function $f$. This was later generalized to the setting of arbitrary abelian groups by Green and Ruzsa \cite{GreenRuzsa2007Freiman}, but many fundamental questions remain open, for example, whether $f$ can be a polynomial.\par 
Another famous problem in additive combinatorics is the Cameron-Erd\H{o}s  conjecture about the number of sum-free subsets of $[n]$, which was solved independently by Green~\cite{green2004cameron} and Sapozhenko~\cite{sapozhenko03:_camer_erd}. More recently Alon, Balogh, Morris and Samotij~\cite{alon2013refinement} obtained a refinement of the Cameron-Erd\H{o}s conjecture  using an early form of the method of hypergraph containers. In order to prove this refinement of the Cameron-Erd\H{o}s conjecture, they needed a bound on the number of $s$-sets $A\subset [n]$ with doubling constant $K$. They moreover conjectured that the following stronger (and, if true, best possible) bound holds.
\begin{conj}[Alon, Balogh, Morris and Samotij]\label{conj:ABMS}
For every $\delta > 0$, there exists $C > 0$ such that the following holds. If $s \geq C \log n$ and if $K \leq s/C$, then there are at most
$$2^{\delta s} \binom{\frac{1}{2} Ks}{s}$$
sets $J \subset [n]$ with $|J| = s$ and $|J + J| \le K|J|$. 
\end{conj}
The conjecture was later confirmed for $K$ constant by Green and Morris \cite{GreenMorris16}; in fact they proved a slightly more general result: for each fixed $K$ and as $s\to \infty$, the number of sets $J\subset [n]$ with $|J|=s$ and $|J+J|\leq K|J|$ is at most
\[2^{o(s)}\binom{\frac{1}{2}Ks}{s}n^{\lfloor K+o(1) \rfloor}.\] The authors of \cite{GreenMorris16} used this result to bound the size of the largest clique in a random Cayley graph and recently the result was also applied by Balogh, Liu, Sharifzadeh and Treglown \cite{BLST2018maxsumfree} to determine the number of maximal sum-free sets in $[n]$.\par
Our main theorem confirms Conjecture \ref{conj:ABMS} for all $K=o(s/(\log n)^3)$.
\begin{thm}\label{thm:mainpr}
Let $s,n$ be integers and $2\leq K\leq o\big(\frac{s}{(\log n)^3}\big)$. 
The number of sets $J\subset [n]$ with $|J|=s$ such that $|J+J|\leq K|J|$ is at most 
 \[ 2^{o(s)}\binom{\frac{1}{2}Ks}{s}.\]
\end{thm}
We will in fact prove stronger bounds on the error term than those stated above, see Theorem \ref{thm:main}. Nevertheless, we are unable to prove the conjecture in the range $K=\Omega (s/(\log n)^3)$, and actually the conjecture is false for a certain range of values of $s$ and $K\gg s/\log n$.
More precisely, for any integers $n,s$, and any positive numbers $K,\epsilon$ with $\min\{s,n^{1/2-\epsilon}\}\geq K\geq \frac{4\log(24C)s}{\epsilon\log n}$, there are at least \[\binom{\frac{n}{2}}{\frac{K}{4}}\binom{\frac{K s}{8}}{s-\frac{K}{4}}\geq \binom{CKs}{s}\] sets $J\subset [n]$ with $|J|=s$ and $|J+J|\leq Ks$.
The construction\footnote{We would like to thank Rob Morris for pointing out this construction.} is very simple: let $P$ be an arithmetic progression of size $Ks/8$ and set $J=J_0\cup J_1$, where $J_0$ is any subset of $P$ of size $s-K/4$, and $J_1$ is any subset of $[n]\setminus P$ of size $K/4$. For convenience we provide the details in the appendix. \par
Our methods also allow us to characterize the typical structure of an $s$-set with doubling constant $K$, and obtain the following result.
\begin{thm}\label{thm:stabilitypr}
Let $s,n$ be integers and $2\leq K\leq o\big( \frac{s}{(\log n)^3}\big)$. For almost all sets $J\subset [n]$ with $|J|=s$ such that $|J+J|\leq K|J|$, there is a set $T\subset J$ such that $J\setminus T$ is contained in an arithmetic progression of size $\frac{1+o(1)}{2}Ks$ and $|T|= o(s)$.
\end{thm}
In the case $s=\Omega (n)$ (and hence $K=O(1)$), this result was proved by Mazur \cite{mazur2015structure}. We will provide better bounds for the error terms in Theorem~\ref{thm:stability}, below.
\subsection{Abelian Groups}
Notice that the doubling constant is defined for finite subsets of any abelian group. So, given a finite subset $Y$ of an abelian group, one might ask: how many subsets of $Y$ of size $s$ with doubling constant $K$ there are? We are also able to provide an answer to this more general question. From now on, fix an arbitrary abelian group $G$ throughout the paper. To state our main result formally in the context of general abelian groups we define, for each positive real number $t$, the quantity $\beta(t)$ to be the size of the biggest subgroup of $G$ of size at most $t$, that is,
\begin{equation} \label{eq:beta}
\beta(t)=\max \big\{ |H|: H\leqslant G,~|H|\leq t \big\}.
\end{equation}
\begin{thm}\label{thm:mainG}
Let $s,n$ be integers, $2\leq K\leq o(\frac{s}{(\log n)^3})$, and $Y\subset G$ with $|Y|=n$. The number of sets $J\subset Y$ with $|J|=s$ such that $|J+J|\leq K|J|$ is at most 
 \[ 2^{o(s)}\binom{\frac{1}{2}(Ks+\beta)}{s},\]
 where $\beta:= \beta((1+o(1))Ks)$.
\end{thm}
Again we will actually prove somewhat stronger (although slightly more convoluted) bounds for Theorem \ref{thm:mainG}, see Theorem \ref{thm:main}. We remark that Theorem~\ref{thm:mainG} implies Theorem~\ref{thm:mainpr}, since the only finite subgroup of $\mathbb{Z}$ is the trivial one, so in this case $\beta (t)=1$ for all $t$. Finally let us remark that Theorem~\ref{thm:mainG} is best possible in many cases.
Indeed suppose for some integers $l,m$, that the largest subgroup $H\leqslant G$ with $|H|\leq m\leq |G|$ is of size $\beta=\frac{m}{2l-1}$, then there are at least 
\[\binom{\frac{m+\beta}{2}}{s}\]
sets $J\subset G$ of size $s$ such that $|J+J|\leq m$. To see this, take an arithmetic progression $P\subset G/H$ of size $l$ (there exists one because of the choice of $H$) and consider $B=P+H$. Since $|B+B|\leq |P+P||H|=m$, for every set $J\subset B$ of size $s$ we have $|J+J|\leq |B+B|\leq m$. Therefore, there are at least \[\binom{\frac{lm}{2l-1}}{s}=\binom{\frac{m+\beta}{2}}{s}\] sets $J\subset B$ of size $s$ with $|J+J|\leq m$. 
\subsection{The method of hypergraph containers}
Before diving into the proof of the main results, let us briefly mention the main tool used in the proof of Theorem \ref{thm:mainpr}. The method of hypergraph containers, introduced by Balogh, Morris and Samotij~\cite{balogh12:_indep_sets_in_hyper} and independently by Saxton and Thomason~\cite{saxton12:_hyper_contain}, has proven to be a very useful tool in counting problems that involve forbidden structures, for a general overview of the method and its applications see \cite{balogh2018method}. More recently, Morris, Samotij and Saxton \cite{morris2018asymmetric} introduced asymmetric containers, a generalization of hypergraph containers for forbidden structures with some sort of asymmetry, and applied the method to give a structural characterization of almost all graphs with a given number of edges free of an induced $C_4$. A variant of the asymmetric container lemma, which follows essentially from a minor modification of the proof in \cite{morris2018asymmetric}, will be our main tool in this article, we give more details in the next section.
\section{The Asymmetric Container Lemma}\label{sec:container lemma}
In this section we will state our main tool and give a brief explanation of how we will apply it to our problem. Let $Y\subset G$, with $|Y|=n$, and observe that when trying to count sets $J\subset Y$ with $|J|=s$ and $|J+J|\leq Ks$, one may instead count sets $J\subset Y$ such that there is a set $I\subset Y$ with $J+J\subset I$ and $|I|\leq Ks$. Keeping this in mind, the following definition will be useful.
\begin{defn}\label{defn:h(a,b)}
Given disjoint copies of $Y+Y$ and $Y$, namely $Y_0,~Y_1$ respectively, and $A\subset Y_0$ and $B\subset Y_1$, we define $\mathcal{H}(A,B)$ to be the hypergraph with vertex set $V(\mathcal{H}(A,B)):=(Y_0\setminus A )\cup B$ and edge set \[E(\mathcal{H}(A,B)):=\big\{(\{c\}, \{a,b\}):c\in Y_0\setminus A,\, a,b\in B,~ a+b=c \big\}.\] 
\end{defn}
Sometimes when $A$ and $B$ are clear from the context we will denote $\mathcal{H}(A,B)$ simply by $\mathcal{H}$. Notice that $\mathcal{H}(A,B)$ is not uniform since there are edges $(\{c\},\{a\})$ corresponding to $a+a=c$, but these will not be a problem. The usefulness of Definition~\ref{defn:h(a,b)} is that now for every pair of sets $(I,J)$ with $J+J\subset I$ we know that $(Y_0\setminus I )\cup J$ doesn't contain any edges of $\mathcal{H}(A,B)$, so $(Y_0\setminus I )\cup J$  would usually be called an independent set, but instead we will call the pair $(I,J)$ independent for convenience. Since we have a method for counting what are usually called independent sets in hypergraphs, and each of those is in correspondence to what we call an independent pair, we can obtain a theorem for counting independent pairs.\par 
To state the main tool in this article we will need to go into some more slightly technical definitions. We first define a useful generalization of uniform hypergraphs, that includes the hypergraph presented in Definition \ref{defn:h(a,b)}. Given disjoint finite sets $V_0,~V_1$ we define an $(r_0,r_1)$-bounded hypergraph $\mathcal{H}$ on the vertex set $V=V_0\cup V_1$ to be a set of edges $E(\mathcal{H})\subset \binom{V_0}{\leq r_0}\times \binom{V_1}{\leq r_1}$. Note that the hypergraph in Definition~\ref{defn:h(a,b)} is $(1,2)$-bounded. Given a pair $(W_0,W_1)\in 2^{V_0}\times 2^{V_1}$, we say $(W_0,W_1)$ \textit{violates} $(e_0,e_1)\in E(\mathcal{H})$ if $e_0\subset V_0\setminus W_0$ and $e_1\subset W_1$. If a set $(W_0,W_1)$ doesn't violate any $(e_0,e_1)\in E(\mathcal{H})$ then we call $(W_0,W_1)$ \textit{independent} with respect to $\mathcal{H}$. Let $\mathcal{F}_{\leq m}(\mathcal{H})\subset 2^{V(\mathcal{H})}$ be the family of independent pairs $(W_0,W_1)$ such that $~|W_0|\leq m$, and observe that for any pair of sets $(I,J)$, with $|I|\leq m$ and $J+J\subset I$, we have $(I,J)\in \mathcal{F}_{\leq m}(\mathcal{H}(\emptyset,Y))$. We define the codegree $d_{(L_0,L_1)}(\mathcal{H})$ of $L_0\subset V_0,~L_1\subset V_1$ to be the size of the set \[\{(e_0,e_1)\in E(\mathcal{H}):~L_0\subset e_0,~L_1\subset e_1\}\] and we define the maximum $(\ell_0,\ell_1)$-codegree of $\mathcal{H}$ to be
\[\Delta_{(\ell_0,\ell_1)}:=\max\{d_{(L_0,L_1)}(\mathcal{H}):~L_0\subset V_0,~L_1\subset V_1,~|L_0|=\ell_0,~|L_1|=\ell_1\}.\]  
With all of this in mind we introduce a variant of the asymmetric container lemma of Morris, Samotij and Saxton \cite{morris2018asymmetric} that we can, once we have suitable supersaturation theorem to check the codegree condition, apply iteratively and prove Theorem \ref{thm:mainpr}.
\begin{thm}
\label{thm:container}
For all non-negative integers $r_0, r_1$, not both zero, and each $R> 0$, the following holds. Suppose that $\mathcal{H}$ is a non-empty $(r_0, r_1)$-bounded hypergraph with $V(\mathcal{H})=V_0\cup V_1$, and $b,~m$, and $q$ are integers with $b\leq \min\{m,|V_1|\}$, satisfying
\begin{equation}\label{deg cond}
\Delta_{(\ell_0,\ell_1)}(\mathcal{H})\leq R \frac{b^{\ell_0+\ell_1-1}}{m^{\ell_0}|V_1|^{\ell_1}}e(\mathcal{H})\left(\frac{m}{q}\right)^{{1}[\ell_0>0]}
\end{equation}
for every pair $(\ell_0,\ell_1)\in \{0,1,\ldots,r_0\}\times\{0,1,\ldots,r_1\}\setminus \{(0,0)\}$. Then there exists a family $\mathcal{S}\subset \binom{V_0}{\leq r_0b}\times\binom{V_1}{\leq r_1b}$ and functions $f\colon \mathcal{S}\to 2^{V_0}\times 2^{V_1}$ and $g\colon \mathcal{F}_{\leq m}(\mathcal{H})\to \mathcal{S}$, such that, letting $\delta=2^{-(r_0+r_1+1)(r_0+r_1)}R^{-1}$:
\begin{enumerate}[label=(\roman*)]
\item If $f(g(I,J))=(A,B)$ with $A\subset V_0$ and $B\subset V_1$, then $A\subset I$ and $J\subset B$.\label{item:container-1}
\item  For every $(A,B)\in f(\mathcal{S})$ either $|A|\geq \delta q$ or $|B|\leq (1-\delta)|V_1|$.\label{item:container-2}
\item If $g(I,J)=(S_0,S_1)$ and $f(g(I,J))=(A,B)$ then $S_0\subset V_0\setminus I$ and $S_1\subset J$, and $|S_0|>0$ only if $|A|\geq \delta q$.\label{item:container-3}
\end{enumerate}
\end{thm}
The proof of this variant of the asymmetric container lemma is virtually identical to that in~\cite{morris2018asymmetric}, but, for the sake of completeness, it is provided in the appendix. Let us remark that the main difference between this statement of the asymmetric container lemma and the one in \cite{morris2018asymmetric} is that we partition the vertex set in two parts and treat them differently, which is essential in our application. More specifically, we will apply the container lemma iteratively in such a way that $V_1$ will shrink much more than $V_0$, and to account for this imbalance we must differentiate between the two sets of the partition. Another small difference is that the hypergraph $\mathcal{H}$ doesn't need to be uniform. Finally we observe that if $S_0$ is non-empty, where $g(I,J)=(S_0,S_1)$, then we must have $|A|\geq \delta q$, where $f(g(I,J))=(A,B)$. 
\section{The Supersaturation Results}
We would like to remind the reader that $G$ will always be a fixed abelian group throughout the paper. To apply Theorem \ref{thm:container} to our setting we will need, for sets $A,B\subset G$, bounds on the number of pairs $(b_1,b_2)\in B\times B$ such that $b_1+b_2\not \in A$. In the case $G=\mathbb{Z}$, one such result is Pollard's Theorem \cite{pollard1974generalisation}, which tell us that if $|B|\geq (1/2+\epsilon)|A|$ and $\epsilon<1/2$ then at least an $\epsilon^2$ proportion of all pairs $(b_1,b_2)\in B\times B$ are such that $b_1+b_2\not \in A$. To prove similar results for arbitrary abelian groups one has to have some control on the structure of the group. With this in mind, we define the following quantity.
\begin{defn}\label{alpha}
Given finite sets $U,V\subset G$, we define \[\alpha(U,V)=\max \big\{|V'|: V'\subset G,~ |V'| \leq  |V|,~|\langle V' \rangle |\leq |U|+|V|-|V'| \big\}.\]
\end{defn}
Given $U,V\subset G$ and $x\in G$ we will use the notation $1_U*1_V(x)$ to denote the number of pairs $(u,v)\in U\times V$ such that $u+v=x$.
The following theorem is the generalization we want of Pollard's theorem for arbitrary abelian groups. It is a simple variant of a result of Hamidoune and Serra \cite{hamidoune2008note}, but for completeness we provide a proof in the appendix.
\begin{thm}\label{thm:hamserra}
Let $t$ be a positive integer and $U,V\subset G$ with
 $t\leq |V|\leq |U|<\infty.$ Then
\begin{equation}
\sum _{x\in G} \min(1_U*1_V(x), t)\ge  t\big(|U|+|V|-t-\alpha\big),\end{equation}
where $\alpha:=\alpha(U,V)$
\end{thm}
This implies the following corollary.
\begin{cor}\label{supersat}
Let $A,B\subset G$ be finite and non-empty sets, let $0<\epsilon<\frac{1}{2}$ and set $\beta:=\beta((1+4\epsilon)|A|)$. If $|B|\geq (\frac{1}{2}+\epsilon)(|A|+\beta)$ then there are at least $\epsilon ^2 |B|^2$ pairs $(b_1,b_2)\in B^2$ such that $b_1+b_2\not \in A.$
\end{cor}
\begin{proof}
Note first that if $|B|\geq (1+\epsilon)|A|$ then the result is trivial, since for each element $a\in A$ there are at most $|B|$ pairs $(b_1,b_2)\in B^2$ with $b_1+b_2=a$, and therefore there are at least $|B|^2-|A||B|\geq \epsilon^2 |B|^2$ pairs in $B$ whose sum is not in $A$. When $|B|\leq (1+\epsilon) |A|$ we will apply Theorem~\ref{thm:hamserra} with $U=V=B$ and $t=\epsilon |B|$. We first observe that 
\[\alpha(B,B)\leq \max\big(\beta, 2|B|-(1+4\epsilon)|A|\big).\]
Indeed, suppose that $B'\subset G$ satisfies $|\langle B' \rangle|\leq 2|B|-|B'|$. If $|\langle B' \rangle|> (1+4\epsilon) |A|$ then $|B'|\leq 2|B|-|\langle B' \rangle|\leq 2|B|-(1+4\epsilon)|A|$. Otherwise, if $|\langle B' \rangle|\leq  (1+4\epsilon) |A|$, then by the definition~\eqref{eq:beta} of $\beta$, we have $|B'|\leq |\langle B' \rangle|\leq \beta$.\par 
Now by Theorem~\ref{thm:hamserra}, we have
\begin{equation*}
\sum _{x\in G} \min(1_B*1_B(x), \epsilon |B|)\geq  \epsilon |B|\Big ((2-\epsilon)|B|-\max\big(\beta, 2|B|-(1+4\epsilon)|A|\big)\Big).
\end{equation*}
By subtracting from both sides the sum over $x\in A$, we obtain  \begin{equation*}\sum _{x\in {G\setminus A}} \min(1_B*1_B(x), \epsilon |B|)\ge  \epsilon |B|\Big((2-\epsilon)|B|-\max\big (\beta, 2|B|-(1+4\epsilon)|A|\big)-|A|\Big).
\end{equation*}
Now, if $2|B|-(1+4\epsilon)|A|\geq \beta$, then, using that $|B|\leq 2|A|$, \[\sum _{x\in {G\setminus A}} 1_B*1_B(x)\geq \epsilon |B|\big( 4\epsilon|A|-\epsilon |B|  \big)\geq \epsilon^2|B|^2\]
as required. Otherwise, if $\beta\geq 2|B|-(1+4\epsilon)|A|$, then \[\sum _{x\in {G\setminus A}} 1_B*1_B(x)\geq \epsilon |B|\big((2-\epsilon)|B|-\beta-|A|\big)\geq \epsilon^2 |B|^2,\] since $|B|\geq (\frac{1}{2}+\epsilon)(|A|+\beta)$ and $0<\epsilon<\frac{1}{2}$, so $(2-\epsilon)-\frac{2}{1+2\epsilon}\geq \epsilon$.
\end{proof}
To prove a stability theorem for almost all sets with a given size and doubling constant we will also need the following result of Mazur \cite{mazur2015structure}.
\begin{thm}\label{thm:mazur}
  Let $l$ and $t$ be positive integers, with $t\leq l/40$, and let $B\subset\mathbb{Z}$ be a set of size $l$. Suppose that
  \begin{equation*}
   \sum_{x\in\mathbb{Z}}\min(1_B*1_B(x),t)\le (2+\delta)lt,
  \end{equation*}
  for some $0<\delta\leq 1/8$. Then there is an arithmetic progression $P$  of length at most $(1+2\delta)l+6t$ containing all but at most $3t$ points of $B$.
 \end{thm}
From Theorem \ref{thm:mazur} we can easily deduce the following corollary:
\begin{cor}\label{satstability}
Let $s$ be an integer, $K>0$, and $0<\epsilon<2^{-10}$. If $A,B\subset \mathbb{Z}$, with $(1-\epsilon)\frac{Ks}{2}\leq|B|\leq (1+2\epsilon)\frac{Ks}{2}$ and $|A|\leq Ks$ then one of the following holds:\begin{enumerate}[label=(\alph*)]
\item \label{item:supersat}
There are at least $4\epsilon^2 K^2s^2$ pairs $(b_1,b_2)\in B^2$ such that $b_1+b_2\not\in A$.
\item \label{item:AP}
There is an arithmetic progression $P$ of size at most $\frac{Ks}{2}+32\epsilon Ks$ containing all but at most $8\epsilon Ks$ points of $B$.
\end{enumerate}
\end{cor}
\begin{proof}
Suppose first that 
 \begin{equation}\label{eq:smallpopdub}
   \sum_{x\in\mathbb{Z}}\min(1_B*1_B(x),t)\le (2+8\epsilon)2\epsilon |B|Ks.
  \end{equation}
In this case we apply Theorem~\ref{thm:mazur} with $l:=|B|$, $\delta:=8\epsilon$, and $t=2\epsilon Ks\leq l/40$, and deduce that~\ref{item:AP} holds. Therefore suppose~\eqref{eq:smallpopdub} doesn't hold, in this case
\begin{equation*}
   \sum_{x\in\mathbb{Z}\setminus A}\min(1_B*1_B(x),t)\geq (2+8\epsilon)(1-\epsilon)\epsilon K^2s^2-t|A|,
  \end{equation*}
since $|B|\geq (1-\epsilon)\frac{1}{2}Ks$. Noting that $t|A|\leq 2\epsilon K^2s^2$ it follows that \[\sum_{x\in\mathbb{Z}\setminus A}1_B*1_B(x)\geq \Big((2+8\epsilon)(1-\epsilon) -2 \Big)\epsilon K^2s^2\geq 4\epsilon^2K^2s^2,\] since $\epsilon< 2^{-10}$, so \ref{item:supersat} holds as required.
\end{proof}
\section{The Number of Sets with a given Doubling}
In this section we prove the following statement which implies Theorems~\ref{thm:mainpr} and \ref{thm:mainG}.
\begin{thm}\label{thm:main}
Let $s,n$ be integers, let $2\leq K< 2^{-36}\frac{s}{(\log n)^3}$, and let $Y\subset G$ with $|Y|=n$. The number of sets $J\subset Y$ with $|J|=s$ such that $|J+J|\leq K|J|$ is at most 
 \[ \exp\Big(2^9\lambda K^{1/6}s^{5/6}\sqrt{\log n}\Big)\binom{\frac{1}{2}(Ks+\beta)}{s},\]
 where $\beta:= \beta\big(Ks+2^6K^{7/6}s^{5/6}\sqrt{\log n}\big)$ and $\lambda:=\min\big\{\frac{K}{K-2},\log s\big\}$.
\end{thm}
Theorem \ref{thm:main} will follow easily from the following container theorem combined with Corollary \ref{supersat}. We will also use it together with Corollary~\ref{satstability} to prove Theorem~\ref{thm:stability}.
\begin{thm}\label{thm:mcontainer}
Let $m,n$ be integers with $m\geq (\log n)^2$, let $Y\subset G$ with $|Y|=n$, and let $0<\epsilon<\frac{1}{4}$. There is a family $\mathcal{A}\subset 2^{Y+Y}\times 2^Y$ of pairs of sets $(A,B)$, of size
\begin{equation}\label{item:Asmall}
|\mathcal{A}|\leq \exp\Big(2^{16}\frac{1}{\epsilon^2}\sqrt{m}(\log n)^{3/2}\Big)
\end{equation}
such that:
\begin{enumerate}[label=(\roman*)]
\item \label{item:everyind}
For every pair of sets $J\subset Y$, $I\subset Y+Y$, with $J+J\subset I$ and $|I|\leq m$ there is $(A,B)\in \mathcal{A}$ such that $A\subset I$ and $J\subset B$. 
\item \label{item:fewedges}
For every $(A,B)\in \mathcal{A}$, $|A|\leq m$ and either $|B|\leq \frac{m}{\log n}$ or there are at most $\epsilon^2|B|^2$ pairs $(b_1,b_2)\in B\times B$ such that $b_1+b_2\not \in A$.
\end{enumerate}
\end{thm}
\begin{proof}[Proof that Theorem \ref{thm:mcontainer} implies  Theorem \ref{thm:main}]\par 
Let $\mathcal{A}$ be a family given by Theorem~\ref{thm:mcontainer} applied with $m:=Ks$ and $\epsilon>0$ to be chosen later. Then by condition \ref{item:everyind}, for every $s$-set $J$ with doubling constant $K$ there is a pair $(A,B)\in \mathcal{A}$ such that $J\subset B$ and $A\subset J+J$. Define $\mathcal{B}$ to be the family of all sets $B$ that are in some container pair, that is \[\mathcal{B}=\{B\subset Y:\exists A\text{ such that }(A,B)\in \mathcal{A}\}.\]
 Observe that, by Corollary \ref{supersat} and condition~\ref{item:fewedges} on $\mathcal{A}$, for every $B\in \mathcal{B}$ we have $|B|\leq (\frac{1}{2}+\epsilon)(m+\beta)$, where $\beta:=\beta((1+4\epsilon)m)$, since the number of pairs $(b_1,b_2)\in B^2$ such that $b_1+b_2\not\in A$ is at most $\epsilon^2|B|^2$ and $\frac{m}{\log n}\leq (\frac{1}{2}+\epsilon)(m+\beta)$. Therefore the number of sets of size $s$ with doubling constant $K$ is at most 
\begin{equation}\label{eq:Kssets1}
|\mathcal{B}|\max_{B\in \mathcal{B}}\binom{|B|}{s}\leq \exp\Big(2^{16}\frac{1}{\epsilon^2}\sqrt{Ks}(\log n)^{3/2}\Big)\binom{(\frac{1+2\epsilon}{2})(Ks+\beta)}{s}.
\end{equation}
Let $\lambda:=\min\{\frac{K}{K-2},\log s\}$, suppose first that $\frac{K}{K-2}\leq \log s$. By applying the inequality $\binom{cn}{k}\leq (\frac{cn-k}{n-k})^k\binom{n}{k}$ with $k=s$, $c=1+2\epsilon$ and $n=\frac{Ks+\beta}{2}$, it follows that in this case (\ref{eq:Kssets1}) is at most
\[\exp\Big(2^{16}\frac{1}{\epsilon^2}\sqrt{Ks}(\log n)^{3/2}+2\epsilon \lambda s\Big)\binom{\frac{Ks+\beta}{2}}{s}.\] Now choosing $\epsilon:=2^4{\big(\frac{K}{s}\big)}^{1/6}\sqrt{\log n}$, by our restrictions on $K$ we see that \[\epsilon<2^4{\Big(\frac{1}{2^{36}(\log n)^3}\Big)}^{1/6}\sqrt{\log n}=\frac{1}{4}.\] It follows that there are at most $\exp{\big(2^9 \lambda K^{1/6}s^{5/6}\sqrt{\log n}\big)}\binom{\frac{1}{2}(Ks+\beta)}{s}$ sets of size $s$ with doubling constant $K$, when $\frac{K}{K-2}\leq \log s$. 
If $\log s\leq \frac{K}{K-2}$ we use the binomial estimate \[\binom{(\frac{1+2\epsilon}{2})(Ks+\beta)}{s}\leq \exp\Big( 4\epsilon s\log\frac{1}{\epsilon}\Big)\binom{\frac{Ks+\beta}{2}}{s}\] and the result follows by a similar calculation. Since $\beta(m+4\epsilon m)= \beta(Ks+2^6K^{7/6}s^{5/6}\sqrt{\log n})$, this proves the theorem.
\end{proof}
Before we proceed with the proof of Theorem \ref{thm:mcontainer}, let us give a brief overview of how we will deduce it from Theorem \ref{thm:container}. We fix from now on a finite subset $Y\subset G$ with $|Y|=n$, and recall that the $(1,2)$-bounded hypergraph $\mathcal{H}(A,B)$ in Definition \ref{defn:h(a,b)} was defined to have as edges pairs $(\{c\},\{a,b\})$ where $a+b=c$, with $a,b\in B$ and $c\not \in A$. Note that condition \ref{item:fewedges} in Theorem \ref{thm:mcontainer} implies that $\mathcal{H}(A,B)$ has at most $\frac{\epsilon^2}{2}|B|^2$ edges, as long as $|B|> \frac{m}{\log n}$. We remind the reader that a pair of sets $I\subset Y+Y$ and $J\subset Y$ with $J+J\subset I$ correspond to an independent set in $\mathcal{H}(A,B)$ for any $A\subset Y+Y$ and $B\subset Y$, since there are no $c\not \in I$ and $a,b\in J$ such that $a+b=c$. If we additionally assume that $(I,J)\in \mathcal{F}_{\leq m}(\mathcal{H})$, then we know that every $J$ that is in such an independent pair satisfies $|J+J|\leq m$.\par 
Our strategy will be to iteratively apply the container lemma until either there are few edges in the hypergraph $\mathcal{H}(A,B)$, or $|A|>m$, in which case the container doesn't contain any elements of $\mathcal{F}_{\leq m}(\mathcal{H})$. More precisely we will build a rooted tree $\mathcal{T}$ with root $\mathcal{H}(\emptyset,Y)$ whose vertices correspond to hypergraphs $\mathcal{H}(A,B)$ and whose  leaves correspond to a family $\mathcal{A}$ satisfying the conclusion of Theorem \ref{thm:mcontainer}. Given a vertex $\mathcal{H}(A,B)$ of the tree, such that $|A|\leq m$, $|B|> \frac{m}{\log n}$ and 
\begin{equation}\label{eq:manyedgestree}
e(\mathcal{H}(A,B))> \frac{\epsilon^2}{2}|B|^2,
\end{equation} 
we will generate its children by applying the following procedure:
\begin{enumerate}[label=(\alph*)]
\item \label{item:checkcont} Apply the asymmetric container lemma (Theorem~\ref{thm:container}) to $\mathcal{H}:=\mathcal{H}(A,B)$ setting $$R:=\frac{2}{\epsilon^2},\quad q:=\frac{m}{\log n},\quad b:=\sqrt{\frac{m}{\log n}}.$$ Notice that the co-degrees of $\mathcal{H}$ satisfy $$\max\big\{\Delta_{(1,0)}(\mathcal{H}),\Delta_{(0,1)}(\mathcal{H})\big\}\leq |B|= \frac{2}{\epsilon^2}\frac{\epsilon^2|B|^2}{2|B|}\leq R\frac{e(\mathcal H)}{|B|}$$ and $$\Delta_{(0,2)}(\mathcal{H})=\Delta_{(1,1)}(\mathcal{H})=\Delta_{(1,2)}(\mathcal{H})= 1= \frac{2}{\epsilon^2}\frac{b^2}{q|B|^2}\frac{\epsilon^2}{2}|B|^2\leq R\frac{b^2}{q|B|^2}e(\mathcal{H}),$$ since \eqref{eq:manyedgestree} holds. Since $b<q<|B|$, it follows that \[\Delta_{(0,2)}(\mathcal{H})\leq R\frac{b^2}{q|B|^2}e(\mathcal{H})\leq R\frac{b}{|B|^2}e(\mathcal{H}),\] \[\Delta_{(1,1)}(\mathcal{H})\leq R\frac{b^2}{q|B|^2}e(\mathcal{H})\leq R\frac{b}{q|B|}e(\mathcal{H})\] and \[\Delta_{(1,0)}(\mathcal{H})\leq R\frac{e(\mathcal H)}{|B|}\leq R\frac{e(\mathcal H)}{q},\] as required.
\item \label{item:propcontain}
By Theorem \ref{thm:container}, there exists a family $\mathcal{C}\subset 2^{(Y+Y)\setminus A}\times 2^{B}$ of at most 
\begin{equation}\label{nchild}
\binom{n^2}{b}\binom{|B|}{2b}\leq n^{4b}\leq e^{4\sqrt{m\log n}},
\end{equation}
pairs of sets $(C,D)$ that satisfies the conditions of the container lemma. That is for each independent pair $(I,J)\in \mathcal{F}_{\leq m}(\mathcal{H})$, with $I\subset Y+Y$ and $J\subset Y$, there is $(C,D)\in \mathcal{C}$ such that $C\subset I$ and $J\subset D$, and either $|C|\geq \delta\frac{m}{\log n}$, or $D\leq (1-\delta)|B|$.
\item For each $(C,D)\in \mathcal{C}$, let $\mathcal{H}(A\cup C,D)$ be a child of $\mathcal{H}(A,B)$ in the tree $\mathcal{T}$.
\end{enumerate}
Now to count the number of leaves of $\mathcal{T}$ we will first bound its depth.
\begin{lemma}\label{deept}
The tree $\mathcal{T}$ has depth at most $d=2^{14}\epsilon^{-2}\log n$.
\end{lemma}
\begin{proof} We will prove that after $d$ iterations either $|A|>m$, $|B|\leq \frac{m}{\log n}$ $e(\mathcal{H}(A,B))\leq \frac{\epsilon^2}{2}|B|^2$. Notice that the $\delta$ provided by Theorem \ref{thm:container} in this application is $2^{-13}\epsilon^2$ and in each iteration either we increase the size of $A$ by $\delta q$ or we decrease the size of $B$ by $\delta |B|$. After $d$ iterations, either we would have increased the size of $A$ more than $\frac{d}{2}$ times, in which case \[|A|> \frac{d}{2}\delta q=\frac{2^{13}\log n}{\epsilon^2}2^{-13}\epsilon^2\frac{m}{\log n}=m,\]
or we would have reduced the size of $B$ at least $\frac{d}{2}$ times, in which case
\[|B|\leq (1-\delta)^{\frac{d}{2}}n< e^{-\frac{\delta d}{2}}n\leq e^{-\log n }n=1.\]
In either case, we would have stopped already by this point because we only generate children of $\mathcal{H}(A,B)$ if $|A|\leq m$, $|B|>\frac{m}{\log n}$ and \eqref{eq:manyedgestree} holds.
\end{proof}
\par 
\begin{proof}[Proof of Theorem \ref{thm:mcontainer}]\par 
Let $\mathcal{L}$ be the set of leaves of the tree $\mathcal{T}$ constructed above, and define \[\mathcal{A}:=\{(A,B):A\subset Y+Y,~B\subset Y, ~\mathcal{H}(A,B)\in \mathcal{L},~|A|\leq m\}.\]
Notice that for every $(A,B)\in \mathcal{A}$, we have either the bound $e(\mathcal{H}(A,B))\leq \frac{\epsilon^2}{2}|B|^2$ or $|B|\leq \frac{m}{\log n}$, since they come from the leaves of $\mathcal{T}$ and $|A|\leq m$. Since the edges of $\mathcal{H}(A,B)$ correspond exactly to pairs $a,b\in B$ such that $a+b\not\in A$, it follows that $\mathcal{A}$ has property \ref{item:fewedges}. \par
To bound the size of $\mathcal{A}$, notice that the number of leaves of the tree $\mathcal{T}$ is at most $Z^d$ where $Z$ denotes the maximum number of children of a vertex of the tree and $d$ denotes its depth. Thus, by \eqref{nchild} and Lemma \ref{deept},
\[|\mathcal{A}|\leq |\mathcal{L}|\leq Z^{d}\leq \exp\Big(2^{16}\frac{1}{\epsilon^2}\sqrt{m}(\log n)^{3/2}\Big),\]
so $\mathcal{A}$ satisfies \eqref{item:Asmall}, as required. \par
Finally, observe that for every pair of sets $J\subset Y,~I\subset Y+Y$ with $J+J\subset I$ and $|I|\leq m$, there is $(A,B)\in \mathcal{A}$ such that $A\subset I$ and $J\subset B$. Indeed $(I,J)\in \mathcal{F}_{\leq m}\big(\mathcal{H}(\emptyset, Y)\big)$ and therefore, by property \ref{item:propcontain} of our containers, there exists a path from the root to a leaf of $\mathcal{T}$ such that $A\subset I$ and $J\subset B$ for every vertex $\mathcal{H}(A,B)$ of the path, so \ref{item:everyind} holds. 
\end{proof}
\section{Typical Structure Result}
In this section we use Theorem \ref{thm:mcontainer} to determine the typical structure of a set $J\subset [n]$ of a given size with doubling constant $K$.
\begin{thm}\label{thm:stability}
Let $s,n$ be integers, let $2\leq K\leq \frac{s}{2^{120}(\log n)^3}$, and let $J\subset [n]$ be a uniformly chosen random set with $|J|=s$ and $|J+J|\leq K|J|$. With probability at least $1-\exp(-K^{1/6}s^{5/6}\sqrt{\log n})$ the following holds: there is a set $T\subset J$, of size $|T|\leq 2^{15} K^{1/6}s^{5/6}\sqrt{\log n}$, such that $J\setminus T$ is contained in an arithmetic progression of size \[\frac{Ks}{2}+2^{17}K^{7/6}s^{5/6}\sqrt{\log n}.\]
\end{thm}
The proof of Theorem \ref{thm:stability} is similar to that of Theorem~\ref{thm:main}, but we use Corollary~\ref{satstability} as well as Corollary~\ref{supersat}.
\begin{proof}[Proof of Theorem \ref{thm:stability}]
Let $G:=\mathbb{Z}$ and apply Theorem \ref{thm:mcontainer} to the set $Y:=[n]$ with $m:=Ks$ and $\epsilon>0$ to be chosen later. We say $B\subset [n]$ is $(\epsilon,Ks)$\textit{-close to an arithmetic progression} if there is an arithmetic progression $P$ with $|P|\leq \frac{Ks}{2}+2^5\epsilon Ks$, and a set $T\subset B$ with $|T|\leq 2^5\epsilon |B|$ such that $B\setminus T \subset P$. We claim that if $\mathcal{A}$ is the family provided by Theorem~\ref{thm:mcontainer}, then for every pair $(A,B)\in \mathcal{A}$ either 
\begin{enumerate}[label=(\Roman*)]
\item \label{item:smallcont} $|B|\leq (1-\epsilon)\frac{Ks}{2}$ or
\item \label{item:contcloseap} $B$ is $(\epsilon,Ks)$-close to an arithmetic progression. 
\end{enumerate}
To see this, note first that, by condition \ref{item:fewedges} in Theorem 4.2, for every pair $(A,B)\in \mathcal{A}$ either there are at most $\epsilon ^2|B|^2$ pairs $b_1,b_2\in B$ with $b_1+b_2\not\in A$ or $|B|\leq \frac{m}{\log n}$, and so, by Corollary \ref{supersat}, $|B|\leq (1+2\epsilon)\frac{Ks}{2}$. Now, if \ref{item:smallcont} doesn't hold, that is $|B|\geq (1-\epsilon)\frac{Ks}{2}$, then, by Corollary \ref{satstability}, \ref{item:contcloseap} holds, since there are at most $\epsilon^2|B|^2< 4\epsilon^2K^2s^2$ pairs $b_1,b_2\in B$ such that $b_1+b_2\not\in A$.\par
Now we will count the number of sets $J$ of size $s$ and doubling constant $K$ such that $J$ is not $(2^4\epsilon,Ks)$-close to an arithmetic progression.
Recall from Theorem \ref{thm:mcontainer} \ref{item:everyind} that, for any such set, there exists $(A,B) \in \mathcal{A}$ such that $J \subset B$. Now, observe that there are at most $|\mathcal{A}|\binom{(1-\epsilon)\frac{Ks}{2}}{s}$ sets $J$ of size $s$ that are contained in a set $B$ such that $(A,B)\in \mathcal{A}$ and $|B|\leq (1-\epsilon)\frac{Ks}{2}$. Choosing $\epsilon :=2^{6}{(\frac{K}{s})}^{1/6}\sqrt{\log n}<2^{-10}$ and using the bound \eqref{item:Asmall} on the size of $\mathcal{A}$, we obtain
\begin{align}\label{eq:multAsmall}
\begin{split}
| \mathcal{A}|\binom{(1-\epsilon)\frac{Ks}{2}}{s} & \leq \exp \big(2^{16}\epsilon^{-2}\sqrt{Ks}(\log n)^{3/2}-\epsilon s\big)\binom{\frac{Ks}{2}}{s}\\ &\leq \exp\big(-2^5K^{1/6}s^{5/6}(\log n)^{1/2}\big)\binom{\frac{Ks}{2}}{s}.
\end{split}
\end{align}\par
Finally we count the number of sets $J$ of size $s$ that are not $(2^4\epsilon,Ks)$-close to an arithmetic progression and are contained in a set $B$ such that $(A,B)\in \mathcal{A}$ and $B$ is $(\epsilon,Ks)$-close to an arithmetic progression. For each such $B$, let $P$ be an arithmetic progression with $|P|\leq \frac{Ks}{2}+2^5\epsilon Ks$, and $T\subset B$ be a set with $|T|\leq 2^5\epsilon |B|\leq 2^5\epsilon Ks$, such that $B\setminus T \subset P$. Observe that, there at most
\begin{equation}\label{eq:farAPfinal}
\sum_{s'\geq 2^9\epsilon s}\binom{(1+2\epsilon)\frac{Ks}{2}}{s-s'}\binom{2^5\epsilon Ks}{s'}
\end{equation}
$s$-sets $J\subset B$ that are not $(2^4\epsilon,Ks)$-close to an arithmetic progression, since they must have $s-s'$ elements in $B\setminus T$ and $s'$ elements in $T$ for some $s'\geq 2^9\epsilon s$. Indeed, otherwise $J\setminus T\subset P$, with $|P|\leq Ks+2^9\epsilon Ks$ and $|J\cap T|<2^9\epsilon |J|$. 
To bound this we will use \[\binom{a}{c-d}\binom{b}{d}\leq \binom{a}{c}\Big(\frac{4bc}{ad}\Big)^d,\]
valid for $d\leq c\leq a/4$. Note that, by our choice of $\epsilon$, we have $|\mathcal{A}|\leq e^{\epsilon s}$ (cf. \eqref{eq:multAsmall}). Hence summing \eqref{eq:farAPfinal} over $(A,B)\in \mathcal{A}$ we obtain\footnote{We remark that if $K<16$ then $\binom{2^5\epsilon Ks}{s'}=0$ for all $s'\geq 2^9\epsilon s$, so we may suppose that $K\geq 16$.}
\begin{align}\label{eq:final}
\begin{split}
|\mathcal{A}| \cdot s\max_{s'\geq 2^9\epsilon s}(1+4\epsilon)^s\binom{\frac{Ks}{2}}{s-s'}\binom{2^5\epsilon Ks}{s'}& \leq |\mathcal{A}|\cdot s\max_{s'\geq 2^9\epsilon s}(1+4\epsilon)^s \binom{\frac{Ks}{2}}{s}\Big(\frac{2^8\epsilon s}{s'}\Big)^{s'} \\ \leq \Big(\frac{2^8\epsilon s}{2^9\epsilon s}\Big)^{2^9\epsilon s}2^{6\epsilon s}\binom{\frac{Ks}{2}}{s} & \leq \exp\big( -2^{11}K^{1/6}s^{5/6}\sqrt{\log n}\big)\binom{\frac{Ks}{2}}{s}.
\end{split}
\end{align}
Finally observe that the bound \eqref{eq:multAsmall} and \eqref{eq:final} imply the probability we claimed in the statement since, by taking all subsets of size $s$ of an arithmetic progression of length $\frac{Ks}{2}$, there are at least $\binom{\frac{Ks}{2}}{s}$ sets of size $s$ and doubling constant $K$.
\end{proof}

\section*{Acknowledgements}
The author would like to thank Rob Morris for his thorough comments on the manuscript and many helpful discussions. We would also like to thank Mauricio Collares, Victor Souza and Natasha Morrison for interesting discussions and comments on early versions of this manuscript. The author is also very grateful to Hoi Nguyen for spotting a mistake in and suggesting various improvements to the first version of this manuscript.

\bibliography{my}{}
\bibliographystyle{ijmart}

\appendix
\section{The proof of Theorem \ref{thm:container}}
In this appendix we provide, for completeness, a proof of Theorem \ref{thm:container}. The proof given below is essentially identical to that presented in~\cite{morris2018asymmetric} with some adaptations to the notation. We would like to thank the authors of~\cite{morris2018asymmetric} for allowing us to reproduce their proof in this appendix.

\subsection{Setup}
Let $r_0$ and $r_1$ be non-negative integers and let $R$ be a positive real. Let $b$, $m$, and $r$ be positive integers and suppose that $\HH$ is a $(r_0, r_1)$-bounded hypergraph\footnote{We remark that from now on all hypergraphs are allowed to have multi-edges, and the edges are counted with multiplicity.} with vertex set $V=(V_0,V_1)$ satisfying~\eqref{deg cond} for each pair $(\ell_0, \ell_1)$ and  $b\leq \min\{m,|V_1|\}$ as in the statement of Theorem~\ref{thm:container}. We claim that, without loss of generality, denoting from now on $v_0(\mathcal{H})=|V_0|$ and $v_1(\mathcal{H})=|V_1|$, we may assume that $m \le v_0(\HH)$. Indeed, if $m > v_0(\HH)$, then we may replace $m$ with $v_0(\HH)$ as $\FF_{\le m} \subseteq \FF(\HH) = \FF_{\le v_0(\HH)}(\HH)$ and the right-hand side of~\eqref{deg cond} is a non-increasing function of $m$.
We shall be working only with hypergraphs whose uniformities come from the set 
\[
  \UU := \big\{ (1,0), (2,0), \ldots, (r_0,0), (r_0,1), \ldots, (r_0,r_1) \big\}.
\]
The maximum codegrees we must check for each uniformity will come from the set
\[
  \VV(i_0,i_1) := \big\{ 0, 1, \ldots, i_0\big\}\times \big\{0, 1, \ldots, i_1 \big\}\setminus \{(0,0)\}.
\]
We now define a collection of numbers that will be upper bounds on the maximum degrees of the hypergraphs constructed by our algorithm. To be more precise, for each $(i_0, i_1) \in \UU$ and all $(\ell_0, \ell_1)\in \VV(i_0,i_1)$, we shall force the maximum $(\ell_0, \ell_1)$-degree of the $(i_0, i_1)$-uniform hypergraph not to exceed the quantity $\Del$, defined as follows.

\begin{defn}
  \label{dfn:Delta}
  For every $(i_0, i_1) \in \UU$ and every $(\ell_0, \ell_1) \in \VV(i_0,i_1)$, we define the number $\Del$ using the following recursion:
  \begin{enumerate}[label=(\arabic*)]
  \item 
    Set $\Delta_{(\ell_0, \ell_1)}^{(r_0,r_1)} := \Delta_{(\ell_0, \ell_1)}(\HH)$ for all $(\ell_0, \ell_1) \in \VV(r_0,r_1)$.\smallskip
  \item
    If $i_0 = r_0$ and $0 \le i_1 < r_1$, then
    \[
      \Del := \max \left\{ 2 \cdot \Delf{}{+1}{}{+1}, \, \frac{b}{v_1(\HH)} \cdot \Delf{}{}{}{+1} \right\}.
    \]
  \item
    If $0 < i_0 < r_0$ and $i_1 = 0$, then
    \[
      \Del := \max \left\{ 2 \cdot \Delf{+1}{}{+1}{}, \, \frac{b}{m} \cdot \Delf{}{}{+1}{} \right\}.
    \]
  \end{enumerate}
\end{defn}

The above recursive definition will be convenient in some parts of the analysis. In other parts, we shall require the following explicit formula for $\Del$, which one easily derives from Definition~\ref{dfn:Delta} using a straightforward induction on $r_0 + r_1 - i_0 - i_1$.

\begin{obs}
  \label{obs:Delta}
  For all $i_0$, $i_1$, $\ell_0$, and $\ell_1$ as in Definition~\ref{dfn:Delta},
  \[
    \Del = \max \left\{ 2^{d_0+d_1} \left(\frac{b}{v_1(\HH)}\right)^{r_1-i_1-d_1} \left(\frac{b}{m}\right)^{r_0-i_0-d_0} \Delta_{(\ell_0+d_0, \ell_1+d_1)}(\HH) \scolon 0 \le d_j \le r_j-i_j \right\}.
  \]
\end{obs}

For future reference, we note the following two simple corollaries of Observation~\ref{obs:Delta} and our assumptions on the maximum degrees of $\HH$, see~\eqref{deg cond}. Suppose that $(i_0, i_1) \in \UU$. If $i_1 > 0$, then necessarily $i_0 = r_0$ and hence,
\begin{equation}
  \label{eq:Delta01}
  \begin{split}
    \Delta_{(0,1)}^{(i_0,i_1)} & \le \max \left\{ 2^{d_1} \left(\frac{b}{v_1(\HH)}\right)^{r_1-i_1-d_1} R \cdot \frac{b^{d_1}}{v_1(\HH)^{d_1+1}} \cdot e(\HH) \scolon 0 \le d_1 \le r_1 - i_1\right\} \\
    & \le 2^{r_1}R \left(\frac{b}{v_1(\HH)}\right)^{r_1-i_1} \frac{e(\HH)}{v_1(\HH)} = 2^{r_1}R \left(\frac{b}{v_1(\HH)}\right)^{r_1-i_1} \left(\frac{b}{m}\right)^{r_0-i_0} \frac{e(\HH)}{v_1(\HH)}.
  \end{split}  
\end{equation}
Moreover, if $i_0 > 0$ and $i_1=0$, then 
\begin{equation}
  \label{eq:Delta10}
  \begin{split}
    \Delta_{(1,0)}^{(i_0,i_1)} & \le \max \left\{ 2^{d_0+d_1} \left(\frac{b}{v_1(\HH)}\right)^{r_1-d_1}\left(\frac{b}{m}\right)^{r_0-i_0-d_0} R \cdot \frac{b^{d_0+d_1}}{m^{d_0} \cdot v_1(\HH)^{d_1}} \cdot \frac{e(\HH)}{q} \right\}\\
    & \le 2^{r_0+r_1}R \left(\frac{b}{v_1(\HH)}\right)^{r_1} \left(\frac{b}{m}\right)^{r_0-i_0} \frac{e(\HH)}{q},
  \end{split}
\end{equation}
where the maximum is over all pairs $(d_0, d_1)$ of integers satisfying $0 \le d_j \le r_j - i_j$.\par
We will build a sequence of hypergraphs with decreasing uniformity, starting with $\HH$, and making sure that, for each hypergraph $\GG$ in the sequence, we have an appropriate bound on its maximum codegrees. To this end we define the following set of pairs with large codegree.
\begin{defn}
  \label{dfn:MDel}
  Given $(i_0, i_1) \in \UU$, $(\ell_0, \ell_1) \in \VV(i_0,i_1)$, and an $(i_0, i_1)$-uniform hypergraph $\GG$, we define
  \[
    \MDel(\GG) = \left\{ (T_0, T_1) \in \binom{V(\GG)}{\ell_0} \times \binom{V(\GG)}{\ell_1} \scolon \deg_\GG(T_0,T_1) \ge \frac{1}{2} \cdot \Del \right\}.
  \]
\end{defn}
Finally, let us say that $c \in \{0,1\}$ is \emph{compatible} with $(i_0, i_1) \in \UU$ if the unique pair $(i_0', i_1')  \in \UU \cup \{(0,0)\}$ with $i_0' + i_1' = i_0 + i_1 - 1$ satisfies $i_c' = i_c - 1$ (and $i_{1-c}' = i_{1-c}$). By the definition of $\UU$, it follows that $c=1$ for $(i_0, i_1) \in \UU$ if and only if $i_1 > 0$.

\subsection{The algorithm}
\label{sec:algorithm}

We shall now define precisely a single round of the algorithm we use to prove the container lemma. To this end, fix some $(i_0, i_1) \in \UU$ and a compatible $c \in \{0,1\}$ and (as in the definition of a compatible $c$) set
\begin{equation}
  \label{eq:i-prime}
  i_c' = i_c - 1 \qquad \text{and} \qquad i_{1-c}' = i_{1-c}.
\end{equation}
Suppose that $\GG$ is an $(i_0, i_1)$-bounded hypergraph with $V(\GG) = V(\HH)$. A single round of the algorithm takes as input an arbitrary $(I,J) \in \FF(\GG)$ and outputs an $(i_0', i_1')$-bounded hypergraph $\GG^*$ satisfying $V(\GG^*) = V(\GG)$ and $(I,J) \in \FF(\GG^*)$ as well as a~set of vertices $S$ of $\GG$ such that $|S|\leq b$ and either $S\subset J$ or $S\subset V_0\setminus I$. Crucially, the number of possible outputs of the algorithm (over all possible inputs $(I,J) \in \FF(\GG)$) is at most $\binom{v_c(\mathcal{H})}{\le b}$.

Assume that there is an implicit linear order $\preccurlyeq$ on $V(\GG)$. The \emph{$c$-maximum vertex} of a hypergraph $\cA$ with $V(\cA) = V(\GG)$ is the $\preccurlyeq$-smallest vertex among those $v$ that maximise $|\{(A_0, A_1) \in \cA \scolon v \in A_c\}|$.

\medskip

\noindent
\textbf{The algorithm.}
Set $\cA^{(0)} := \GG$, let $S$ be the empty set, and let $\GGn^{(0)}$ be the empty $(i_0',i_1')$-bounded hypergraph on $V(\GG)$. Do the following for each integer $j \ge 0$ in turn:
\begin{enumerate}[label={(S\arabic*)}]
\item
  \label{item:alg-stop}
  If $|S| = b$ or $\cA^{(j)}$ is empty, then set $J := j$ and \STOP.
\item
  \label{item:alg-c-maximum}
  Let $v_j\in V_c$ be the $c$-maximum vertex of $\cA^{(j)}$.
\item
  \label{item:alg-main-step}
  If $c=0$ and $v_j\not\in I$ or $c=1$ and $v_j\in J$, then add $j$ to the set $S$ and let
  \[
    \GGn^{(j+1)} := \GGn^{(j)} \cup \Big\{ \big( A_0 \setminus \{v_j\},  A_1 \setminus \{v_j\}\big) \scolon (A_0, A_1) \in \cA^{(j)} \text{ and } v_j \in A_c \Big\}.
  \]
\item
  \label{item:alg-cleanup}
  Let $\cA^{(j+1)}$ be the hypergraph obtained from $\cA^{(j)}$ by removing from it all pairs $(A_0, A_1)$ such that either of the following hold:
  \begin{enumerate}[label={(\alph*)}]
  \item
    \label{item:cleanup-1}
    $v_j \in A_c$;
  \item
    \label{item:cleanup-2}
    there exist $T_0 \subseteq A_0$ and $T_1 \subseteq A_1$, not both empty, such that 
    \[
      (T_0, T_1) \in \MDelp\big( \GGn^{(j+1)} \big)
    \]
 for some $(\ell_0,\ell_1)\in \mathcal{V}(i_0',i_1')$.
  \end{enumerate}
\end{enumerate}
Finally, set $\cA := \cA^{(L)}$ and $\GGn := \GGn^{(L)}$. Moreover, set
\[
  W := \big\{ 0, \dotsc, L-1 \big\} \setminus S = \Big\{ j \in \big\{ 0, \dotsc, L-1 \big\} \scolon v_j\not\in V_0\setminus I \textsl{ and }v_j\not \in J\Big\}.
\]

\smallskip

Observe that the algorithm always stops after at most $v(\GG)$ iterations of the main loop. Indeed, since all constraints $(A_0, A_1)$ with $v_j \in A_c$ are removed from $\cA^{(j+1)}$ in part~\ref{item:cleanup-1} of step~\ref{item:alg-cleanup}, the vertex $v_j$ cannot be the $c$-maximum vertex of any $\cA^{(j')}$ with $j' > j$ and hence the map $\{0, \dotsc, L-1\} \ni j \mapsto v_j \in V(\GG)$ is injective. 

\subsection{The analysis}

We shall now establish some basic properties of the algorithm described in the previous subsection. To this end, let us fix some $(i_0, i_1) \in \UU$ and a compatible $c \in \{0,1\}$ and let $i_0'$ and $i_1'$ be the numbers defined in~\eqref{eq:i-prime}. Moreover, suppose that $\GG$ is an $(i_0, i_1)$-bounded hypergraph and that we have run the algorithm with input $(I,J) \in \FF(\GG)$ and obtained the $(i_0', i_1')$-bounded hypergraph $\GGn$, the integer $L$, the injective map $\{0, \dotsc, L-1\} \ni j \mapsto v_j \in V(\GG)$, and the partition of $\{0, \dotsc, L-1\}$ into $S$ and $W$ such that $v_j\in J$ or $v_j\in V_0\setminus I$ if and only if $j \in S$. We first state two straightforward, but fundamental, properties of the algorithm.

\begin{obs}
  \label{obs:h-in-FF-Gn}
  If $(I,J) \in \FF(\GG)$, then $(I,J) \in \FF(\GGn)$.
\end{obs}

\begin{proof}
  Observe that $\GGn$ contains only constraints of the form:
  \begin{enumerate}[label={(\textit{\roman*})}]
  \item
    \label{item:constraint-i}
    $(A_0 \setminus \{v\}, A_1)$, where $v \in A_0$ and $v\in V_0\setminus I$, or
  \item
    \label{item:constraint-ii}
    $(A_0, A_1 \setminus \{v\})$, where $v \in A_1$ and $v\in J$,
  \end{enumerate}
  where $(A_0, A_1) \in \GG$, see~\ref{item:alg-main-step}. Hence, if $(I,J)$ violated a constraint of type~\ref{item:constraint-i} (resp.~\ref{item:constraint-ii}) then $(I,J)$ would also violate the constraint $(A_0, A_1)$, as $v\in V_0\setminus I$ (resp.\ $v\in J$).
\end{proof}

The next observation says that if the algorithm applied to two pairs $(I,J)$ and $(I',J')$ outputs the same set $\{v_j \scolon j \in S\}$, then the rest of the output is also the same. 

\begin{obs}
  \label{obs:number-of-containers}
Fix the hypergraph $\GG$ we input in the algorithm, suppose that the algorithm applied to $(I',J') \in \FF(\GG)$ outputs a hypergraph~$\GGn'$, an integer $L'$, a map $j \mapsto v_j'$, and a partition of $\{0, \dotsc, L'-1\}$ into $S'$ and $W'$. If $\{v_j \scolon j \in S\} = \{v_j' \scolon j \in S'\}$, then $\GGn = \GGn'$, $L = L'$, $v_j = v_j'$ for all $j$, and $W = W'$.
\end{obs}

\begin{proof}
  The only step of the algorithm that depends on the input pair $(I,J)$ is~\ref{item:alg-main-step}. There, an index $j$ is added to the set $S$ if and only if $v_j\in V_0\setminus I$ or $v_j\in J$. Therefore, the execution of the algorithm depends only on the set $\{v_j \scolon j \in S\}$ and the hypergraph $\GG$.
\end{proof}

The next two lemmas will allow us to maintain suitable upper and lower bounds on the degrees and densities of the hypergraphs obtained by applying the algorithm iteratively. The first lemma, which is the easier of the two, states that if all the maximum degrees of $\GG$ are appropriately bounded, then all the maximum degrees of $\GGn$ are also appropriately bounded.

\begin{lemma}\label{lemma:alg-analysis-degrees}
Given $(\ell_0, \ell_1) \in \VV(i_0,i_1)$ and $\ell_c > 0$, set $\ell_c' = \ell_c-1$ and $\ell_{1-c}' = \ell_{1-c}$. If $\Delta_{(\ell_0,\ell_1)}(\GG) \le \Del$, then $\Delta_{(\ell_0',\ell_1')}(\GGn) \le \Delp$.
\end{lemma}

\begin{proof}
Suppose (for a contradiction) that there exist sets $T_0'$ and $T_1'$, with $|T_0'| = \ell_0'$ and $|T_1'| = \ell_1'$, such that $\deg_{\GGn}(T_0', T_1') > \Delp$. Let $j$ be the smallest integer satisfying
  \[
    \deg_{\GGn^{(j+1)}}(T_0', T_1') > \frac{1}{2} \cdot \Delp
  \]
  and note that $j \ge 0$, since $\GGn^{(0)}$ is empty. We claim first that
   \begin{equation}
    \label{eq:deg-j-is-deg-final}
    \deg_{\GGn} (T_0', T_1') = \deg_{\GGn^{(j+1)}}(T_0', T_1').
  \end{equation}
  Indeed, observe that $(T_0', T_1') \in M^{(i_0',i_1')}_{(\ell_0',\ell_1')} \big( \GGn^{(j+1)} \big)$, and therefore the algorithm removes from $\cA^{(j)}$ (when forming $\cA^{(j+1)}$ in step~\ref{item:alg-cleanup}) all pairs $(A_0, A_1)$ such that $T_0' \subseteq A_0$ and $T_1' \subseteq A_1$. As a consequence, no further pairs $(A_0', A_1')$ with $T_0' \subseteq A_0'$ and $T_1' \subseteq A_1'$ are added to $\GGn$ in step~\ref{item:alg-main-step}.

  We next claim that
  \begin{equation}
    \label{eq:one-step-deg-change}
    \deg_{\GGn^{(j+1)}}(T_0', T_1') - \deg_{\GGn^{(j)}}(T_0', T_1') \le \Del.
  \end{equation}
  To see this, recall that when we extend $\GGn^{(j)}$ to $\GGn^{(j+1)}$ in step~\ref{item:alg-main-step}, we only add pairs $\big( A_0 \setminus \{v_j\}, A_1 \setminus \{v_j\} \big)$ such that $(A_0, A_1) \in \cA^{(j)} \subseteq \GG$ and $v_j \in A_c$. Therefore, setting $T_c = T_c' \cup \{v_j\}$ and $T_{1-c} = T_{1-c}'$, we have
  \begin{equation*}
    \deg_{\GGn^{(j+1)}}(T_0', T_1') - \deg_{\GGn^{(j)}}(T_0', T_1') \le \deg_{\GG}(T_0, T_1) \le \Delta_{(\ell_0, \ell_1)}(\GG) \le \Del,   
  \end{equation*}
  where the last inequality is by our assumption, as claimed.

  Combining~\eqref{eq:deg-j-is-deg-final} and~\eqref{eq:one-step-deg-change}, it follows immediately that
  \[
    \deg_{\GGn}(T_0',T_1') \le \frac{1}{2} \cdot \Delp + \Del \le \Delp,
  \]
  where the final inequality holds by Definition~\ref{dfn:Delta}. This contradicts our choice of $(T_0', T_1')$ and therefore the lemma follows.
\end{proof}

We are now ready for the final lemma, which is really the heart of the matter. We will show that if $\GG$ has sufficiently many edges and all of the maximum degrees of $\GG$ are appropriately bounded, then either the output hypergraph~$\GGn$ has sufficiently many edges, or we either have a big set $W\subset V_1\setminus J$, or we have a big set $W\subset I$. We remark that here we shall use the assumption that $|I|\leq m$.

\begin{lemma}
  \label{lemma:alg-analysis-progress}
  Suppose that $|I| \le m$ and let $\alpha > 0$. If 
  \begin{enumerate}[label=(A\arabic*)]
  \item
    \label{item:assumption-edges}
    $e(\GG) \ge \alpha \cdot \big( \frac{b}{v_1(\HH)} \big)^{r_1-i_1} \big( \frac{b}{m} \big)^{r_0-i_0} e(\HH)$ and\smallskip
     \item
    \label{item:assumption-Delta}
    $\Delta_{(\ell_0, \ell_1)}(\GG) \le \Del$ for every $(\ell_0, \ell_1) \in \VV(i_0,i_1)$,\smallskip
  \end{enumerate}
  then at least one of the following statements is true:
  \begin{enumerate}[label=(P\arabic*)]
  \item
    \label{item:reduce-uniformity}
    $e(\GGn) \ge 2^{-i_0-i_1-1} \alpha \cdot \big( \frac{b}{v_1(\HH)} \big)^{r_1-i_1'} \big( \frac{b}{m} \big)^{r_0-i_0'} e(\HH)$.\smallskip
  \item
    \label{item:determine-many-0}
    $c = 1$ and $|W| \ge 2^{-r_1-1} R^{-1} \alpha \cdot v_1(\HH)$.\smallskip
  \item
    \label{item:determine-many-1}
    $c = 0$ and $|W| \ge 2^{-r_0-r_1-1} R^{-1} \alpha  \cdot q$.
  \end{enumerate}
\end{lemma}

\begin{proof}
  Suppose first that $c = 0$ and observe that\footnote{Recall that $\GGn$ (and $\GGn^{(j)}$ etc.) are multi-hypergraphs and that edges are counted with multiplicity.}
  \begin{equation}
    \label{eq:eG-sum}
    e(\GGn) = \sum_{j \in S} \left( e(\GGn^{(j+1)}) - e(\GGn^{(j)}) \right) = \sum_{j \in S} \deg_{\cA^{(j)}}(\{v_j\}, \emptyset),
  \end{equation}
  since $e(\GGn^{(j+1)}) - e(\GGn^{(j)}) = \deg_{\cA^{(j)}}(\{v_j\}, \emptyset)$ for each $j \in S$ and $\GGn^{(j+1)} = \GGn^{(j)}$ for each $j \not\in S$. To bound the right-hand side of~\eqref{eq:eG-sum}, we count the edges removed from $\cA^{(j)}$ in \ref{item:cleanup-1} and \ref{item:cleanup-2} of step \ref{item:alg-cleanup}, which gives
  \[
    e(\cA^{(j)}) - e(\cA^{(j+1)}) \le \deg_{\cA^{(j)}}(\{v_j\}, \emptyset) + \sum_{(\ell_0, \ell_1)} \big| \MDelp(\GGn^{(j+1)}) \setminus \MDelp(\GGn^{(j)}) \big| \cdot \Delta_{(\ell_0, \ell_1)}(\GG).
  \]
  Summing over $j \in \{0, \ldots, L-1\}$, it follows (using~\eqref{eq:eG-sum}) that
  \[
    e(\GG) - e(\cA) \le e(\GGn) + |W| \cdot \Delta_{(1,0)}(\GG) + \sum_{(\ell_0, \ell_1)} \big| \MDelp(\GGn) \big| \cdot \Del,
  \]
  since $\cA = \cA^{(L)} \subseteq \ldots \subseteq \cA^{(0)} =\GG$ and $\Delta_{(\ell_0, \ell_1)}(\GG) \le \Del$ by~\ref{item:assumption-Delta}. Observe also that if $c = 1$, then we obtain an identical bound, with $\Delta_{(1,0)}(\GG)$ replaced by $\Delta_{(0,1)}(\GG)$.

  In order to discuss both cases simultaneously, we set $\chi(0) = (1,0)$ and $\chi(1) = (0,1)$. Observe that
  \begin{equation}
    \label{eq:e-G0-increment-estimate}
    \Delta_{\chi(c)}(\cA) \le \Delta_{\chi(c)}(\cA^{(j)}) \le \Delta_{\chi(c)}(\GG) \le \Delta_{\chi(c)}^{(i_0,i_1)},
  \end{equation}
  since $\cA \subseteq \cA^{(j)} \subseteq \GG$ and $\GG$ satisfies~\ref{item:assumption-Delta}. It follows that, for both $c \in \{0, 1\}$,
  \begin{equation}
    \label{eq:Gn-final-0}
    e(\GG) - e(\cA) \le e(\GGn) + |W| \cdot \Delta_{\chi(c)}^{(i_0,i_1)} + \sum_{(\ell_0, \ell_1)} \big|\MDelp(\GGn) \big| \cdot \Del.
  \end{equation}
  Now, recall that $v_j$ is the $c$-maximum vertex of $\cA^{(j)}$ and observe that therefore, by~\eqref{eq:eG-sum} and~\eqref{eq:e-G0-increment-estimate},  
  \begin{equation}
    \label{eq:e-Gn-S}
    e(\GGn) = \sum_{j \in S} \Delta_{\chi(c)}\big(\cA^{(j)}\big) \ge |S| \cdot \Delta_{\chi(c)}(\cA) = b \cdot \Delta_{\chi(c)}(\cA),
  \end{equation}
  where the equality is due to the fact that $|S| \neq b$ only when $\cA$ is empty, see step~\ref{item:alg-stop}. 
  
 Next, to bound the sum in~\eqref{eq:Gn-final-0}, observe that, by Definition~\ref{dfn:MDel}, we have
 \[
   \big| \MDelp(\GGn) \big| \cdot \frac{1}{2} \cdot \Delnp \le \sum_{|T_0|=l_0, |T_1|=l_1} \deg_{\GGn}(T_0, T_1) \leq  \binom{i_0'}{\ell_0} \binom{i_1'}{\ell_1} \cdot e(\GGn)
 \]
 for each $(\ell_0, \ell_1)\in \VV(i_0',i_1')$ and therefore
 \begin{equation}
   \label{eq:e-Gn-MDel-final}
   \begin{split}
     \sum_{(\ell_0, \ell_1)\in \VV(i_0',i_1')} \big| \MDelp(\GGn) \big| \cdot \Del & \le 2 \cdot \sum_{(\ell_0, \ell_1)} \binom{i_0'}{\ell_0} \binom{i_1'}{\ell_1} \cdot e(\GGn) \cdot \left(\Del / \Delnp\right)  \\
     & \le 2 \cdot \big( 2^{i_0'+i_1'} - 1 \big) \cdot e(\GGn) \cdot \max_{(\ell_0, \ell_1)} \left\{\Del / \Delnp\right\}.
   \end{split}
 \end{equation}
 We claim that $\Del / \Delnp \le m / b$ if $c = 0$ and $\Del / \Delnp \le v_1(\HH) / b$ if $c = 1$. Indeed, both inequalities following directly from Definition~\ref{dfn:Delta}, since if $c = 0$, then $(i_0', i_1') = (i_0-1, i_1)$, and if $c = 1$, then $(i_0', i_1') = (i_0, i_1-1)$. We split the remainder of the proof into two cases, depending on the value of $c$.

  Suppose first that $c = 1$ and observe that substituting~\eqref{eq:e-Gn-MDel-final} into~\eqref{eq:Gn-final-0} yields, using the bound $\Del / \Delnp \le v_1(\HH) / b$,
  \begin{equation}
    \label{eq:Gn-final-case1}
    e(\GG) - e(\cA) \le e(\GGn) + |W| \cdot \Delta_{(0,1)}^{(i_0,i_1)} + 2 \cdot \big( 2^{i_0'+i_1'} - 1 \big) \cdot e(\GGn) \cdot \frac{v_1(\HH)}{b}.
  \end{equation}
  Moreover, by~\eqref{eq:e-Gn-S}, and since $i_1 \ge 1$ when $c = 1$, we have
  \begin{equation}
    \label{eq:Delta-estimates-1}
    \frac{e(\GGn)}{b} \ge \Delta_{(0,1)}(\cA) \ge \frac{i_1 \cdot e(\cA)}{v_1(\cA)} \ge \frac{e(\cA)}{v_1(\HH)},
  \end{equation}
  since the maximum degree of a hypergraph is at least as large as its average degree. Combining~\eqref{eq:Gn-final-case1} and~\eqref{eq:Delta-estimates-1}, we obtain
  \begin{equation}
    \label{eq:eGG-eGGn-W-1}
    \begin{split}
      e(\GG) & \le e(\GGn) \cdot \frac{v_1(\HH)}{b} \cdot \left(\frac{b}{v_1(\HH)} + 1 + 2^{i_0'+i_1'+1} - 2\right) + |W| \cdot \Delta_{(0,1)}^{(i_0,i_1)} \\
      & \le e(\GGn) \cdot \frac{v_1(\HH)}{b} \cdot 2^{i_0+i_1} + |W| \cdot \Delta_{(0,1)}^{(i_0,i_1)},
    \end{split}
  \end{equation}
  since $b \le v_1(\HH)$. Now, if the first summand on the right-hand side of~\eqref{eq:eGG-eGGn-W-1} exceeds $e(\GG) / 2$, then~\ref{item:assumption-edges} implies~\ref{item:reduce-uniformity}, since $(i_0', i_1') = (i_0, i_1-1)$. Otherwise, the second summand is at least $e(\GG)/2$ and by~\ref{item:assumption-edges} and~\eqref{eq:Delta01},
  \[
    |W| \ge \frac{e(\GG)}{2 \cdot \Delta_{(0,1)}^{(i_0,i_1)}} \ge \frac{\alpha}{2^{r_1+1}R} \cdot v_1(\HH),
  \]
  which is~\ref{item:determine-many-0}.

  The case $c = 0$ is slightly more delicate; in particular, we will finally use our assumption that $|I| \le m$. Observe first that if $c = 0$, then $i_1=0$ and substituting~\eqref{eq:e-Gn-MDel-final} into~\eqref{eq:Gn-final-0} yields, using the bound $\Del / \Delnp \le m / b$,
 \begin{equation}
    \label{eq:Gn-final-case2}
    e(\GG) - e(\cA) \le e(\GGn) + |W| \cdot \Delta_{(1,0)}^{(i_0,i_1)} + \big( 2^{i_0+i_1} - 2 \big) \cdot e(\GGn) \cdot \frac{m}{b},
  \end{equation}
  cf.~\eqref{eq:Gn-final-case1}. We claim that
  \begin{equation}
    \label{eq:Delta-estimates-0}
    \frac{e(\GGn)}{b} \ge \Delta_{(1,0)}(\cA) \ge \frac{e(\cA)}{m}.
  \end{equation}
 The first inequality follows from~\eqref{eq:e-Gn-S}, so we only need to prove the second inequality. To do so, observe that $\GG$ is an $(i_0, 0)$-uniform hypergraph (since $c = 0$) and therefore for each pair $(I,J)\in \FF(\GG)$ we must have $I\cap A_0\neq \emptyset$ for every $(A_0, \emptyset) \in \GG$. Now, recall that $(I,J) \in \FF(\GG)$, that $\cA \subseteq \GG$, and that $|I|\leq m$. It follows that $e(\cA) \le m \cdot \Delta_{(1,0)}(\cA)$, as claimed.
  
 Combining~\eqref{eq:Gn-final-case2} and~\eqref{eq:Delta-estimates-0}, we obtain (cf.~\eqref{eq:eGG-eGGn-W-1})
  \begin{equation}
    \label{eq:eGG-eGGn-W-0}
    \begin{split}
      e(\GG) & \le e(\GGn) \cdot \frac{m}{b} \cdot \left(\frac{b}{m} + 1 + 2^{i_0 + i_1} - 2\right) + |W| \cdot \Delta_{(1,0)}^{(i_0,i_1)} \\
      & \le e(\GGn) \cdot \frac{m}{b} \cdot 2^{i_0+i_1} + |W| \cdot \Delta_{(1,0)}^{(i_0,i_1)},
    \end{split}
  \end{equation}
  since $b \le m$. Now, if the first summand on the right-hand side of~\eqref{eq:eGG-eGGn-W-1} exceeds $e(\GG) / 2$, then~\ref{item:assumption-edges} implies~\ref{item:reduce-uniformity}, since $(i_0', i_1') = (i_0-1, i_1)$. Otherwise, the second summand is at least $e(\GG)/2$ and by~\ref{item:assumption-edges} and~\eqref{eq:Delta10},
  \[
    |W| \ge \frac{e(\GG)}{2 \cdot \Delta_{(1,0)}^{(i_0,i_1)}} \ge \frac{\alpha}{2^{r_0+r_1+1}R} \cdot q,
  \]
  which is~\ref{item:determine-many-1}.
\end{proof}

\subsection{Construction of the container}
\label{sec:constr-container}

In this section, we present the construction of containers for pairs in $\FF_{\le m}(\HH)$ and analyse their properties, thus proving Theorem~\ref{thm:container}. For each $s \in \{0, \ldots, r_0+r_1\}$, define
\[
  \alpha_s = 2^{-s(r_0+r_1+1)} \qquad \text{and} \qquad \beta_s = \alpha_s \cdot \left(\frac{b}{v_1(\HH)}\right)^{\min\{{r_1,s\}}}\left(\frac{b}{m}\right)^{\max\{0,s-r_1\}}.
\]
Given an $(I,J) \in \FF_{\le m}(\HH)$, we construct the container $(A,B)$ for $(I,J)$ using the following procedure.

\medskip
\noindent
\textbf{Construction of the container.}
Let $\HH^{(r_0,r_1)} = \HH$, let $S_0 = S_1 = \emptyset$, and let $(i_0, i_1) = (r_0, r_1)$. Do the following for $s = 0, \ldots, r_0+r_1-1$:
\begin{enumerate}[label=(C\arabic*)]
\item
  Let $c \in \{0,1\}$ be the number that is compatible with $(i_0, i_1)$ and let $(i_0', i_1')$ be the pair defined by $i_c' = i_c-1$ and $i_{1-c}' = i_{1-c}$.
\item
  \label{item:constr-run-algorithm}  
  Run the algorithm with $\GG \leftarrow \HH^{(i_0, i_1)}$ to obtain the $(i_0', i_1')$-uniform hypergraph $\GGn$, the sequence $v_0, \ldots, v_{L-1} \in V(\HH)$, and the partition $\{0, 1, \ldots, L-1\} = S \cup W$.
\item
  \label{item:constr-update-signature}
  Let $S_c \leftarrow S_c \cup \{v_j \scolon j \in S\}$.
\item
  \label{item:constr-determine-many}
  If $e(\GGn) < \beta_{s+1} \cdot e(\HH)$, then define $(A,B)$, the container for $(I,J)$, by
  \[
    (A,B)=(W,\emptyset)
  \]
  if $c=0$ and
  \[(A,B)=(\emptyset,V_1\setminus W)\]
  if $c=1$, 
  and \STOP.
\item
  \label{item:constr-reduce-uniformity}
  Otherwise, let $\HH^{(i_0', i_1')} \leftarrow \GGn$ and $(i_0, i_1) \leftarrow (i_0', i_1')$ and \texttt{CONTINUE}.
\end{enumerate}

We will show that the above procedure indeed constructs containers for $\FF_{\le m}(\HH)$ that have the desired properties. To this end, we first claim that for each pair $(i_0, i_1) \in \UU \cup \{(0,0)\}$, the hypergraph $\HH^{(i_0,i_1)}$, if it was defined, satisfies:
\begin{enumerate}[label=(\textit{\roman*})]
\item
  \label{item:container-propty-1}
  $(I,J) \in \FF(\HH^{(i_0,i_1)})$ and
\item
  \label{item:container-propty-2}
  $\Delta_{(\ell_0,\ell_1)}(\HH^{(i_0,i_1)}) \le \Del$ for every $(\ell_0, \ell_1) \in \VV(i_0,i_1)$.
\end{enumerate}
Indeed, one may easily prove~\ref{item:container-propty-1} and~\ref{item:container-propty-2} by induction on $(r_0+r_1) - (i_0+i_1)$. The basis of the induction is trivial as $\HH^{(r_0,r_1)} = \HH$, see Definition~\ref{dfn:Delta}. The inductive step follows immediately from Observation~\ref{obs:h-in-FF-Gn} and Lemma~\ref{lemma:alg-analysis-degrees}.

Second, we claim that for each input $(I,J) \in \FF_{\le m}(\HH)$, step~\ref{item:constr-determine-many} is called for some $s$ and hence the container $ (A,B)$  is defined. If this were not true, the condition in step~\ref{item:constr-reduce-uniformity} would be met $r_0+r_1$ times and, consequently, we would finish with a non-empty $(0,0)$-uniform hypergraph $\HH^{(0,0)}$, i.e., we would have $(\emptyset, \emptyset) \in \HH^{(0,0)}$. But this contradicts~\ref{item:container-propty-1}, since pair satisfies the empty constraint and thus $(I,J) \not\in \FF(\HH^{(0,0)})$. 

Suppose, therefore, that step~\ref{item:constr-determine-many} is executed when $\GG = \HH^{(i_0,i_1)}$ for some $(i_0,i_1) \in \UU$, and note that $s = (r_0+r_1) - (i_0+i_1)$. We claim that $e(\HH^{(i_0,i_1)}) \ge \beta_s e(\HH)$. Indeed, this is trivial if $s = 0$, whereas if $s > 0$ and this were not true, then we would have executed step~\ref{item:constr-determine-many} at the previous step. We therefore have
\[
  e(\GG) = e(\HH^{(i_0,i_1)}) \ge \beta_s \cdot e(\HH) \qquad \text{and} \qquad e(\GGn) < \beta_{s+1} \cdot e(\HH),
\]
which, by Lemma~\ref{lemma:alg-analysis-progress} and~\ref{item:container-propty-2}, implies that either~\ref{item:determine-many-0} or~\ref{item:determine-many-1} of Lemma~\ref{lemma:alg-analysis-progress} holds. Note that if $c = 1$, then $r_1 \ge i_1 > 0$ and we have
\[
  |W| \ge 2^{-r_1-1}R^{-1} \alpha_s \cdot v_1(\HH) \ge \alpha_{r_0+r_1}R^{-1} v_1(\HH) = \delta v_1(\HH),
\]
where $\delta = 2^{-(r_0+r_1)(r_0+r_1+1)} R^{-1}$. On the other hand, if $c = 0$, then $r_0 \ge i_0 > 0$ and  
\[
  |W| \ge 2^{-r_0-r_1-1}R^{-1} \alpha_s \cdot q \ge \alpha_{r_0+r_1}R^{-1} q = \delta q.
\]
This verifies that $(A,B)$ satisfies property~\ref{item:container-2} from the statement of Theorem~\ref{thm:container}.

To complete the proof, we need to show that $(A,B)$ can be assigned to each $(I,J)$ by a pair of functions $f \circ g$ for some $g \colon \FF_{\le m}(\HH) \to \binom{V_0}{\le r_0 b} \times \binom{V_1}{\le r_1 b}$ and to verify that properties~\ref{item:container-1} and~\ref{item:container-3} from the statement of the theorem hold. We claim that one may take $g(I,J) = (S_0, S_1)$, where $S_0$ and $S_1$ are the sets constructed by the above procedure, see~\ref{item:constr-update-signature}. To this end, it suffices to show that if for some $(I,J), (I',J') \in \FF(\HH)$ the above procedure produces the same pair $(S_0, S_1)$, then $f \circ g(I,J)= f \circ g(I',J')$. To see this, observe first that the set $S$ defined in step~\ref{item:constr-run-algorithm} is precisely the set of all indices $j \in \{0, \ldots, L-1\}$ that satisfy $v_j \in S_c$. Indeed, the former set is contained in the latter by construction, see~\ref{item:constr-update-signature}. The reverse inclusion holds because 
\[
  S = \big\{ j \in \{ 0, \dotsc, L-1 \} \scolon v_j\in V_0\setminus I \textsl{ or }v_j\in J \big\}
\]
which is exactly the condition on $v \in S_c$. By Observation~\ref{obs:number-of-containers}, it follows that the output of the algorithm depends only on the pair $(S_0,S_1)$ and hence $(A,B)$, as claimed.

Finally, observe that $S_0 \subseteq V_0\setminus I$ and $S_1 \subseteq J$, by construction, $A\subset I$ and $J\subset B$. This verifies properties~\ref{item:container-1} and~\ref{item:container-3} and hence completes the proof of Theorem~\ref{thm:container}. \qed
\section{Proof of Theorem \ref{thm:hamserra}}
\begin{proof}[Proof of Theorem \ref{thm:hamserra}]\par 
Given an abelian group $G$ and finite subsets $A,B\subset G$, we will proceed by induction on $|B|$ to show that
\begin{equation*}
\sum _{x\in G} \min(1_A*1_B(x), t)\ge  t(|A|+|B|-t-\alpha),
\end{equation*}
for all integers $t\leq |B|$, where $\alpha:=\alpha(A,B)$. First, note that if $t=|B|=1$ then we have \[\sum _{x\in G} \min(1_A*1_B(x), t)=|A|\geq t(|A|-\alpha).\]\par 
Now take $B$ of size $|B|\geq 2$, and define $B'=B-b$ for some $b\in B$ and note that $0\in B'$. Suppose first that $B'+A\subset A$, and observe that in this case $A$ is an union of cosets of $\left< B'\right>$, that is, $A=\bigcup_{i=1}^k \left< B'\right> +h_i$ for some $h_1,...,h_k\in G$. It follows that $1_A*1_{B'}(x)\geq t$ for all $x\in A$, since if $x\in A\cap (\langle B'\rangle+h_i)$ then there are at least $|B'|\geq t$ sums $a+b'=x$ with $a\in A\cap (\left< B'\right> +h_i)$ and $b'\in B'$. Since $G=G-b$ it follows that
\begin{equation*}
\sum _{x\in G} \min(1_A*1_B(x), t)\ge  t|A|\ge  t(|A|+|B|-t-\alpha),
\end{equation*}
where the second inequality follows because $|\langle B'\rangle |\leq |A|$ and so $\alpha(A,B)\geq |B'|=|B|$.\par 
On the other hand, if $B'+A\not \subset A$ then there exists $a^*\in A$ such that $B^*=a^*+B'\not \subset A$ and therefore $1\leq |A\cap B^*|<|B|$. Define $C=A\cup B^*$, $D=A\cap B^*$, $A_1=A\setminus D$ and $B_1=B^*\setminus D$. Note that $1_A=1_{A_1}+1_D$ and $1_{B^*}=1_{B_1}+1_D$ and therefore, by the distributivity property of the convolution operation, 
\begin{multline}\label{eq:splitsum}
1_A*1_{B^*}(x)=(1_{A_1}+1_D)*(1_{B_1}+1_D)(x)\\=1_{A_1}*1_{B_1}(x)+(1_{A_1}+1_{B_1}+1_D)*1_D(x)=1_{A_1}*1_{B_1}(x)+1_C*1_D(x).
\end{multline}
In particular $1_A*1_{B^*}(x)\geq 1_C*1_D(x)$. If $|D|\ge t$ then by applying our induction hypothesis to $C$ and $D$, we obtain
\begin{equation*}
\sum _{x\in G} \min(1_A*1_B(x), t)\ge \sum _{x\in G} \min(1_C*1_D(x), t)\geq  t(|A|+|B|-t-\alpha),
\end{equation*}
where the first step follows since $G=G+a^*$, and the last step follows from the fact that $|C|+|D|=|A|+|B|$ and $\alpha(C,D)\leq \alpha(A,B)$, since $|D|\leq |B|$.\par 
Finally, if $|D|<t$, observe that
\begin{equation}\label{eq:indfinalcsesupersat}
\sum _{x\in G} \min(1_A*1_B(x), t)\geq \sum _{x\in G} \min(1_{A_1}*1_{B_1}(x), t-|D|)+\sum _{x\in G} \min(1_C*1_D(x), |D|),
\end{equation}
by \eqref{eq:splitsum}.
Because $|B_1|<|B|$ we can apply the induction hypothesis to $A_1$ and $B_1$, so the right hand side of \eqref{eq:indfinalcsesupersat} is at least 
\[ (t-|D|)\big(|A|+|B|-|D|-t-\alpha(A_1,B_1)\big)+|C||D|.\]
Noting that $\alpha(A_1,B_1)\leq \alpha(A,B)$, because $|B_1|\leq |B|$, and that $|A|+|B|-|D|=|C|$, it follows that the last expression is at least $t(|A|+|B|-t-\alpha),$ as required. 
\end{proof}

\section{Proof of Lower Bound for large $K$}
\begin{prop}
Let $n$ and $s$ be positive integers, and let $K,\epsilon>0$ and $C\geq 2$ satisfy $\min\{s,n^{1/2-\epsilon}\}\geq K\geq 4\frac{\log(24C)s}{\epsilon\log n}$. There are at least \[\binom{CKs}{s}\] sets $J\subset [n]$ with $|J|=s$ and $|J+J|\leq Ks$.
\end{prop}
\begin{proof}
Choose $P$ to be an arithmetic progression of length $\frac{Ks}{8}$ and let $J=J_0\cup J_1$, with $J_0\subset P$ of size $s-\frac{K}{4}$ and $J_1\subset [n]\setminus P$ of size $\frac{K}{4}$. Then $J$ has doubling constant $K$ since \begin{align*}
|J+J| & \leq |J_0+J_0|+|J_0+J_1|+|J_1+J_1| \\ & \leq 2|P|+|J_0||J_1|+|J_1|^2 \leq \frac{Ks}{4}+\frac{Ks}{4}+\frac{K^2}{16} \leq Ks.
\end{align*}
Finally, by using that $\log (\frac{n}{K^2})\geq \epsilon\log n$ and the bounds \[\binom{b}{d}\binom{a}{c-d}\geq \Big(\frac{bc}{4ad}\Big)^{d}\binom{a}{c}\qquad  \text{and}\qquad a\binom{b}{c}\geq \binom{\frac{a^{1/c}b}{e}}{c}\] valid for  any positive integers $a,b,c,d$, such that $4d\leq c$, we have at least \[\binom{\frac{n}{2}}{\frac{K}{4}}\binom{\frac{Ks}{8}}{s-\frac{K}{4}}\geq \Big(\frac{n}{K^2}\Big)^{K/4} \binom{\frac{Ks}{8}}{s}\geq \binom{\exp( \frac{\epsilon K\log n}{4s})\frac{Ks}{8e}}{s}\] choices for $J$. In particular if $K\geq \frac{4\log(24C)s}{\epsilon\log n}$ this is at least $\binom{CKs}{s}$. 
\end{proof}
\end{document}